\documentclass[dvips,a4paper,11pt]{article}

\usepackage{array}
\usepackage{dcolumn}
\usepackage[dvips]{graphicx}
\usepackage{amsopn}
\usepackage{amssymb,amsfonts,amsbsy}
\usepackage{psfrag}
\usepackage{latexsym}

\usepackage{array}
\usepackage{dcolumn}

\usepackage{epsfig}
\usepackage{enumerate,cite,color}
\usepackage{psfrag}
\usepackage{graphics}
\usepackage{amsmath,amsfonts,amsthm}
\usepackage[cp850]{inputenc}

\usepackage{fancyhdr,sectsty}
\usepackage{slashbox}

\usepackage{url,hyperref}

\font\reffont=cmcsc9

\newcommand{\norm}[1]{\left\Vert#1\right\Vert}

\newtheorem{theorem}{Theorem}

\theoremstyle{definition}
\newtheorem{obs}{Observation}

\title{Olver's asymptotic method: a special case}
\author{\\Chelo Ferreira$^1$, Jos\'{e} L. L\'{o}pez$^2$ and Ester P\'{e}rez Sinus\'{\i}a$^1$\\\\
\small{$^1$ \textsf{\textit{Dpto. de Matem\'{a}tica Aplicada, IUMA, Universidad de Zaragoza}}}\\
\small{ \textsf{\textit{e-mail: cferrei@unizar.es,
ester.perez@unizar.es}}}\\
\small{$^2$
\textsf{\textit{Dpto. de Ingenier\'{\i}a Matem\'{a}tica e Inform\'{a}tica, Universidad P\'{u}blica de Navarra and BIFI, Zaragoza}}}\\
\small{ \textsf{\textit{ e-mail: jl.lopez@unavarra.es}}}}
\date{}

\begin{document}
\normalsize \maketitle

\begin{abstract}
We consider the asymptotic method designed by F. Olver [Olver, 1974] for linear differential equations of the second order containing a large (asymptotic) parameter $\Lambda$: $x^my''-\Lambda^2y=g(x)y$, with $m\in\mathbb{Z}$ and $g$ continuous. Olver studies in detail the cases $m\ne 2$, specially the cases $m=0,\pm 1$, giving the Poincar\'e-type asymptotic expansion of two independent solutions of the equation. The case $m=2$ is different, as the behavior of the solutions for large $\Lambda$ is not of exponential type, but of power type. In this case, Olver's theory does not give as many details as it gives in the cases $m\ne 2$. Then, we consider here the special case $m=2$. We propose two different techniques to handle the problem: (i) a modification of Olver's method that replaces the role of the exponential approximations by power approximations and (ii) the transformation of the differential problem into a fixed point problem from which we construct an asymptotic sequence of functions that converges to the unique solution of the problem.  Moreover, we show that this second technique may also be applied to nonlinear differential equations with a large parameter.
 \\\\
\noindent \textsf{2010 AMS \textit{Mathematics Subject
Classification:} 34A12; 41A58; 41A60; 34B27.
} \\\\
\noindent  \textsf{Keywords \& Phrases:} Second order differential equations. Asymptotic expansions. Green's functions. Banach's fixed point theorem.
\\\\
\end{abstract}

\section{Introduction}

The most famous asymptotic method for second order linear differential equations containing a large parameter is, no doubt, Olver's method. In \cite[Chaps. 10, 11, 12]{olver}, Olver considers the differential equation
\begin{equation}\label{olvereq}
u''-{\tilde\Lambda^2\over z^m}u=h(z)u,\qquad \tilde\Lambda\to\infty,
\end{equation}
with $m=0,-1,1$, $\tilde\Lambda$ is a complex parameter, $z$ is a complex variable and $h$ is an analytic function in a certain region of the complex plane. Correspondingly to these three different $m-$cases, Olver divides the study of \eqref{olvereq} in three canonical cases, say I, II and III, analyzed in Chapters 10, 11 and 12 respectively. In Case I, Olver completes the theory developed in the well-known Liouville-Green approximation, giving a rigorous meaning to the approximation and providing error bounds for the expansions of solutions of \eqref{olvereq} for $m=0$. In Cases II and III, Olver extends the theory introduced in Case I considering, respectively, the case $m=-1$ (differential equations with a turning point) and the case $m=1$ (differential equations with a regular singular point).

In \cite[Chap. 12, Sec. 14]{olver} we can also find indications about the generalization of the study of the asymptotics of the solutions of \eqref{olvereq} for general $m\in\mathbb{Z}$, except $m=2$. In summary, we have that for any $m\in\mathbb{Z}\setminus\lbrace 2\rbrace$, two independent solutions of \eqref{olvereq} have the form
\begin{equation}\label{expan}
u(z)=P_m(z)\sum_{k=0}^{n-1}{A_k(z)\over\tilde\Lambda^{2k}}+{1\over\tilde\Lambda^2}P'_m(z)\sum_{k=0}^{n-1}{B_k(z)\over\tilde\Lambda^{2k}}+R_{m,n}(z),
\end{equation}
where $R_{m,n}(z)={\cal O}(\tilde\Lambda^{-2n})$ uniformly for $z$ in a certain region in the complex plane. In this formula, $P_m(z)$ is one of the two following basic solutions of \eqref{olvereq}, that is, independent solutions of \eqref{olvereq} for $h=0$:
\begin{equation}\label{besels}
P_m(z):=\begin{cases}\sqrt{z}I_{\hat m}(2\hat m\tilde\Lambda z^{1/(2\hat m)}), \\\vspace*{-0.4cm}\\ \sqrt{z}K_{\hat m}(2\hat m\tilde\Lambda z^{1/(2\hat m)}),
\end{cases} \qquad \hat m:={1\over 2-m}.
\end{equation}
In this formula and in the remaining of the paper, the symbols $I_\nu(z)$ and $K_\nu(z)$ denote the principal values of the modified Bessel functions. For example, for $m=0,1,3,4,5,\ldots$, the coefficients $A_k$ and $B_k$ are given by the following system of recurrences: $A_0(z)=1$ and
\begin{equation}\nonumber
\begin{split}
B_n(z)=& {z^{m/2}\over 2}\int z^{m/2}[h(z)A_n(z)-A''_n(z)]dz, \\
A_{n+1}(z)=& -{1\over 2}B_n'(z)+{1\over 2}\int h(z)B_n(z)dz,
\end{split} \hspace*{3cm} n=0,1,2,\ldots
\end{equation}
Both families of coefficients $A_n$ and $B_n$ are analytic at $z=0$ when $h(z)$ is also analytic there. Olver's important contribution is the proof of the asymptotic character of the two expansions \eqref{expan}$-$\eqref{besels} and the derivation of error bounds for the remainder $R_{m,n}(z)$.

For large $\tilde\Lambda$ and fixed $z$, both solutions have an asymptotic behavior of exponential type \cite[Sec. 10.30(ii)]{nist}:
$$
\sqrt{z}I_{\hat m}(2\hat m\tilde\Lambda z^{1/(2\hat m)})\sim{z^{m/4}\over\sqrt{\tilde\Lambda}}e^{2\vert\hat m\Re(\tilde\Lambda z^{1/(2\hat m)})|}, \qquad
\sqrt{z}K_{\hat m}(2\hat m\tilde\Lambda z^{1/(2\hat m)})\sim{z^{m/4}\over\sqrt{\tilde\Lambda}}e^{-2\hat m\Re(\tilde\Lambda z^{1/(2\hat m)})},
$$
both valid in the sector $\vert$Arg$(\tilde\Lambda z^{1/(2\hat m)})\vert<3\pi/2$. Therefore, for any $m\ne 2$, two independent solutions of \eqref{olvereq} have an exponential asymptotic behavior for large $\tilde\Lambda$ and fixed $z$. The above approximations obviously fail for $m=2$. This case is considered by Olver in \cite[Chap. 6, Sec. 5.3]{olver}, where he gives the first order asymptotic approximation (WKB approximation) for two independent solutions of \eqref{olvereq}. Also, in \cite[Chap. 10, Sec. 4.1]{olver}, Olver gives some indications about the derivation of a complete asymptotic expansion in terms of the expansion given for the case $m=0$, although details are not given there.

The purpose of this paper is to analyze the asymptotic behavior of the solutions of the equation $u''-\tilde\Lambda^2z^{-2}u=h(z)u$ in detail. To this end, in the next section we introduce an appropriate change of unknown in the differential equation. In Section 3 we use a fixed point theorem and the Green function of an auxiliary initial value problem to derive an asymptotic as well as convergent expansion of a couple of independent solutions of the equation in terms of iterated integrals of $h(z)$; this technique is based on our previous investigations \cite{lopez}, \cite{lopezdos}, \cite{lopeztres}. In Section 4 we generalize this technique to nonlinear problems, where we obtain an asymptotic expansion of an initial value problem for a nonlinear equation. In Section 5 we use Olver's techniques to obtain asymptotic expansions, of Poincar\'e-type, of two independent solutions of the equation, different from those obtained in Section 3. Section 6 contains and example and some numerical experiments and Section 7 a few remarks and conclusions.

\section{Preliminaries}

Consider the differential equation \eqref{olvereq} with $m=2$. For later convenience, we define the function $g(z):=zh(z)$ and a new large parameter
\begin{equation}\label{lambda}
\Lambda:={1+\sqrt{4\tilde\Lambda^2+1}\over 2}.
\end{equation}
In terms of this parameter and the new function $g(z)$, equation \eqref{olvereq} with $m=2$ reads
\begin{equation}\label{ecu}
z^2 u''(z)-\Lambda(\Lambda-1)u(z)=z g(z)u(z).
\end{equation}
Because this equation is invariant under the transformation $\Lambda\to\ 1-\Lambda$, in the remaining of the paper, and without loss of generality, we consider $\Re\Lambda>1/2$. As we mentioned in the introduction, the general formula \eqref{expan} is not directly applicable to this equation; for  $m=2$, the index $\hat m$ of the basic Bessel functions approximants in \eqref{besels} becomes infinite, the asymptotic behavior of the solutions of \eqref{ecu} is not exponential in $\Lambda$. On the other hand, as it is explained in \cite{lopezliouville}, when we consider this equation with an initial condition at the point $z=0$, a fixed point technique does not work either: the exponent $m=2$ in the coefficient $z^2$ of $u''$ makes the iterated integrals related to the fixed point iterations divergent at $z=0$.

Both problems may be overcome by means of an appropriate change of unknown $u\to y$ that modifies the exponent $m=2$. In order to perform the appropriate change of unknown, we consider here the Frobenius theory.
When the function $g(z)$ is analytic at $z=0$, the exponents of the Frobenius solutions of the differential equation \eqref{ecu} at the regular singular point $z=0$ are $\mu_1=\Lambda$ and $\mu_2=1-\Lambda$. Therefore, two independent solutions of this equation behave, at $z=0$, as $z^{\Lambda}$ and $z^{1-\Lambda}$ respectively.
This fact suggests the following change of unknown: $u\to y:=z^{-\Lambda}u$. The new unknown $y$ satisfies the differential equation
\begin{equation}\label{ecuacion}
zy''(z)+2\Lambda y'(z)=g(z)y(z).
\end{equation}
When $g(z)$ is an analytic function at $z=0$ we know, from Frobenius theory, that this equation has two independent solutions that behave, at $z=0$, as $1$ and $z^{1-2\Lambda}$ respectively. Therefore, in the linear two-dimensional space of solutions of this equation, only one ray of solutions is bounded at $z=0$. These facts determine the kind of possible well-posed problems for this equation. A well-posed initial value problem for the differential equation \eqref{ecuacion} with initial datum given at $z=0$ is
\begin{equation}\label{problemuno}
\begin{cases}
zy''(z)+2\Lambda y'(z)=g(z)y(z)\quad \text{in $\cal D$}, \\
y(0)=y_0,
\end{cases}
\end{equation}
where $y_0$ is any complex parameter, $y_0=\mathcal{O}(1)$ as $\Lambda\to\infty$, and ${\cal D}$ is a star-like domain (bounded or unbounded) in the complex plane centered at $z=0$. In the next section we will show that this problem has a unique solution and we will obtain an asymptotic approximation of the unique solution of this problem. In order to derive an asymptotic expansion of a second independent solution of \eqref{ecuacion} we must consider an initial value problem with initial conditions prescribed at another point $z_0\in{\cal D}$, $z_0\ne 0$:
\begin{equation}\label{problemdos}
\begin{cases}
zy''(z)+2\Lambda y'(z)=g(z)y(z) \quad \text{in ${\cal D}$}, \\
y(z_0)=\bar y_0, \quad y'(z_0)=y_1,
\end{cases}
\end{equation}
where $\bar y_0$ and $y_1$ are complex parameters with , $\bar y_0=\mathcal{O}(1)$ and $y_1=\mathcal{O}(\Lambda)$ as $\Lambda\to\infty$. The existence and uniqueness of solution of this problem follows from Frobenius theory (when $g(z)$ is analytic at $z=z_0$) or from Picard-Lindelof's theorem (when $g(z)$ is continuous at $z=z_0$). In the following, $y_+(z)$ and $y_-(z)$ denote, respectively, the unique solutions of problems \eqref{problemuno} and \eqref{problemdos}.

When we undo the above mentioned change of unknowns, we find that $u_\pm(z):=z^\Lambda y_\pm(z)$ are a couple of independent solutions of \eqref{ecu} whenever $(u_+(z_0),u_+'(z_0))\ne (u_-(z_0), u_-'(z_0))$. Problem \eqref{problemuno} for $y_+$ is equivalent to the following problem for $u_+$:
\begin{equation}\nonumber
\begin{cases}
z^2 u_+''(z)-\Lambda(\Lambda-1)u_+(z)=z g(z)u_+(z)\quad \text{in ${\cal D}$}, \\
\displaystyle\lim_{z\to 0}[z^{-\Lambda}u_+(z)]=y_0,
\end{cases}
\end{equation}
problem that has a unique solution $u_+(z)$.
Problem \eqref{problemdos} for $y_-$ is equivalent to the following problem for $u_-$:
\begin{equation}\nonumber
\begin{cases}
z^2 u_-''(z)-\Lambda(\Lambda-1)u_-(z)=z g(z)u_-(z)\quad \text{in ${\cal D}$},\\
z_0^{-\Lambda}u_-(z_0)=\bar y_0, \quad \displaystyle \lim_{z\to z_0}[z^{-\Lambda}u_-(z)]'=y_1,
\end{cases}
\end{equation}
problem that has a unique solution $u_-(z)$.

In the following section, for each problem, we design a sequence of functions that converges to the unique solution of the problem. For each problem, that sequence has the property of being an asymptotic sequence (not of Poincar\'e-type) for large $\Lambda$. In Section 5 we apply Olver's method to equation \eqref{ecuacion} and find an asymptotic expansion of Poincar\'e-type of two independent solutions of this equation.

\section{A fixed point method}

The unique solution of the initial value problem
\begin{equation}\label{fiuno}
\begin{cases}
z\phi''(z)+2\Lambda \phi'(z)=0\quad \text{in ${\cal D}$,} \\
\phi(0)=y_0,
\end{cases}
\end{equation}
is $\phi_+(z):=y_0$. And the unique solution of the problem
\begin{equation}\label{fidos}
\begin{cases}z\phi''(z)+2\Lambda\phi'(z)=0 \quad \text{in ${\cal D}$,} \\
\phi(z_0)=\bar y_0, \quad \phi'(z_0)=y_1,
\end{cases}
\end{equation}
is
\begin{equation}\label{fi}
\phi_-(z):=\bar y_0+y_1{z_0\over 1-2\Lambda}\left[\left({z\over z_0}\right)^{1-2\Lambda}-1\right].
\end{equation}

After the change of unknown $y_\pm(z)\to w_\pm(z)=y_\pm(z)-\phi_\pm(z)$, and using \eqref{fiuno} and \eqref{fidos}, we find that problems \eqref{problemuno} and \eqref{problemdos} read, respectively,
\begin{equation}\label{problemunobis}
\begin{cases}
zw_+''(z)+2\Lambda w_+'(z)=F_+(z,w_+):=g(z)[w_+(z)+\phi_+(z)] \quad \text{in ${\cal D}$,}  \\
w_+(0)=0,
\end{cases}
\end{equation}
and
\begin{equation}\label{problemdosbis}
\begin{cases}
zw_-''(z)+2\Lambda w_-'(z)=F_-(z,w_-):=g(z)[w_-(z)+\phi_-(z)] \quad \text{in ${\cal D}$,} \\
w_-(z_0)=w_-'(z_0)=0.
\end{cases}
\end{equation}
For convenience, we restrict the differential equations in both problems, \eqref{problemunobis} and \eqref{problemdosbis} (and hence \eqref{problemuno} and \eqref{problemdos}), to an open straight segment ${\cal L}\subset{\cal D}$ (that may be unbounded if ${\cal D}$ is unbounded) with $z=0$ as an end point. Moreover, for problem \eqref{problemdosbis}, $z_0\in{\cal L}$ and $\vert z\vert<\vert z_0\vert$. See Figure 1 below.

\begin{figure}[h!]
\begin{center}
\includegraphics[width=0.7\textwidth]{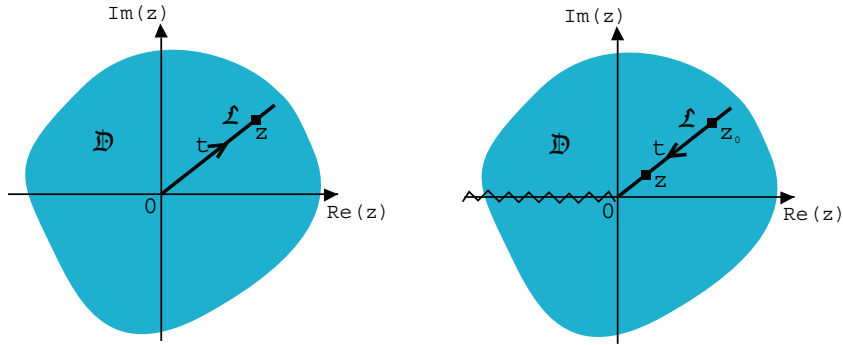}
\end{center}\vspace*{-1.75cm}
\caption{\small Domains ${\cal D}$ and integration paths associated to the respective problems \eqref{problemuno} and \eqref{problemdos}. In both problems, the kernel of the operators ${\bf T}$ and $\tilde{\bf T}$ is bounded by $2$.}
\label{fig:ejem3}
\end{figure}

For the first problem, we seek for solutions of the equation ${\bf L}_+[w_+]:=zw_+''+2\Lambda w'_+-F_+(z,w_+)$ in the Banach space ${\cal B}_+:=\lbrace w_+:{\cal L}\to\mathbb{C}$, $w_+(0)=0\rbrace$. For the second problem, we seek for solutions of the equation ${\bf L}_-[w_-]:=zw_-''+2\Lambda w'_--F_-(z,w_-)$ in the Banach space ${\cal B}_-:=\lbrace w_-:{\cal L}\to\mathbb{C}$, $w_-(z_0)=0\rbrace$. Both spaces equipped with the {\it sup} norm:
$$
\vert\vert w_\pm\vert\vert_\infty:=\sup_{z\in{\cal L}}\vert w_\pm(z)\vert.
$$
We write the equation ${\bf L}_\pm[w_\pm]=0$ in the form ${\bf L}_\pm[w_\pm]={\bf M}[w_\pm]-F_\pm(z,w_\pm)$, with ${\bf M}[w]:=zw''+2\Lambda w'$. Then we solve the equation ${\bf L}_\pm[w_\pm]=0$ for $w_\pm$ using Green's function $G_\pm(z,t)$ of the operator ${\bf M}$ with the appropriate initial conditions \cite{stakgold}. For problem \eqref{problemunobis}, $G_+(z,t)$ is the unique solution of the problem
$$
\begin{cases}
zG_{zz}+2\Lambda G_z=\delta(z-t)\quad \text{in $\mathcal{L}$,}\\
G(0,t)=0, \quad t\in{\cal L}.
\end{cases}
$$
It is given by
$$
G_+(z,t)={1\over 2\Lambda-1}\left[1-\left({t\over z}\right)^{2\Lambda-1}\right]\chi_{[0,z]}(t),
$$
where $\chi_{[0,z]}(t)$ is the characteristic function of the interval $[0,z]$.
For problem \eqref{problemdosbis}, $G_-(z,t)$ is the unique solution of the problem
$$
\begin{cases}
zG_{zz}+2\Lambda G_z=\delta(z-t)\quad \text{in ${\cal L}$}, \\
G(z_0,t)=G_z(z_0,t)=0, \quad t,z_0\in{\cal L}.
\end{cases}
$$
It is given by
$$
G_-(z,t)={1\over 2\Lambda-1}\left[1-\left({t\over z}\right)^{2\Lambda-1}\right]\chi_{[z,z_0]}(t).
$$
Then, any solution $w_+(z)$ of \eqref{problemunobis} is a solution of the Volterra integral equation $w_+(z)=[{\bf T}w_+](z)$, and any solution $w_-(z)$ of \eqref{problemdosbis} is a solution of the Volterra integral equation $w_-(z)=[{\bf T}w_-](z)$, where the integral operator ${\bf T}$ is defined by
\begin{equation}\nonumber
[{\bf T}w_\pm](z):= {1\over 2\Lambda-1}\int_{z_0}^z\left[1-\left({t\over z}\right)^{2\Lambda-1}\right]g(t)[w_\pm(t)+\phi_\pm(t)]dt,
\end{equation}
where $z_0$ must be set equal to zero for $w_+$. For later convenience, in the case of $w_-$ we need to define a {\it rescaled} unknown $\tilde w_-(z):=z^{2\Lambda-1}w_-(z)$ and consider the {\it rescaled} operator
\begin{equation}\nonumber
[{\bf \tilde T}\tilde w_-](z):= {1\over 2\Lambda-1}\int_{z_0}^z\left[\left({z\over t}\right)^{2\Lambda-1}-1\right]g(t)[\tilde w_-(t)+\tilde\phi_-(t)]dt,
\end{equation}
with $\tilde\phi_-(z):=z^{2\Lambda-1}\phi_-(z)$.

For any complex $z$ in ${\cal L}$, the kernel $1-(t/z)^{2\Lambda-1}$ of ${\bf T}$, is uniformly bounded in $t\in[0,z]$ by $2$, independently of $\Lambda$ and $z$. Also, for any complex $z$ in ${\cal L}$, with $\vert z\vert<\vert z_0\vert$, the kernel $(z/t)^{2\Lambda-1}-1$ of ${\bf \tilde T}$, is uniformly bounded in $t\in[z,z_0]$ by $2$, independently of $\Lambda$ and $z$.

From the Banach fixed point theorem \cite[pp. 26, Theorem 3.1]{bailey} it is well-known that, if any power of the operator ${\bf T}$ is contractive in ${\cal B}_+$, then the equation $w_+(z)=[{\bf T} w_+](z)$ has a unique solution $w_+(z)$ (fixed point of ${\bf T}$) and the sequence $w_{n+1}^+=[{\bf T} w_n^+]$, $w_0^+=0$, converges to that solution $w_+(z)$. Analogously, if any power of the operator ${\bf \tilde T}$ is contractive in ${\cal B}_-$, then the equation $\tilde w_-(z)=[{\bf \tilde T} \tilde w_-](z)$ has a unique solution $\tilde w_-(z)$ (fixed point of $\tilde{\bf T}$) and the sequence $\tilde w_{n+1}^-=[{\bf \tilde T} \tilde w_n^-]$, $\tilde w_0^-=0$, converges to that solution $\tilde w_-(z)$.

We show this for the operator $\tilde{\bf T}$. The proof for the operator ${\bf T}$ is identical replacing $z_0$ by $0$.  It is straightforward to show the contractive character of the operator $\tilde{\bf T}$: from its definition we have that, for any couple $u$, $v\in{\cal B}_-$,
$$
\vert[\tilde{\bf T}u](z)-[\tilde{\bf T}v](z)\vert\le {2\over \vert 2\Lambda-1\vert }\int_{z_0}^z\vert g(t)\vert\vert u(t)-v(t)\vert\vert dt\vert\le \left\vert{2(z-z_0)\over 2\Lambda-1}\right\vert\,\vert\vert g\vert\vert_\infty\,\vert\vert u-v\vert\vert_\infty.
$$
We also have
\begin{equation}\nonumber
\begin{split}
\vert[\tilde{\bf T}^2u](z)-[\tilde{\bf T}^2v](z)\vert&\le {2\over \vert 2\Lambda-1\vert}\int_{z_0}^z\vert g(t)\vert\vert [\tilde{\bf T}u](t)-[\tilde{\bf T}v](t)\vert\vert dt\vert \\&\le \left\vert{[2(z-z_0)]^2\over 2(2\Lambda-1)^2}\right\vert\,\vert\vert g\vert\vert_\infty^2\,\vert\vert u-v\vert\vert_\infty
\end{split}
\end{equation}
and
\begin{equation}\nonumber
\begin{split}
\vert[\tilde{\bf T}^3u](z)-[\tilde{\bf T}^3v](z)\vert &\le {2\over \vert 2\Lambda-1\vert}\int_{z_0}^z\vert g(t)\vert\vert [\tilde{\bf T}^2u](t)-[\tilde{\bf T}^2v](t)\vert\vert dt\vert\\ &\le \left\vert{2(z-z_0)]^3\over 3!(2\Lambda-1)^3}\right\vert\,\vert\vert g\vert\vert_\infty^3\,\vert\vert u-v\vert\vert_\infty.
\end{split}
\end{equation}
It is straightforward to prove, by means of induction over $n$ that, for $n=1,2,3$,\ldots,
\begin{equation}\label{tn}
\vert[\tilde{\bf T}^nu](z)-[\tilde{\bf T}^nv](z)\vert\le \left\vert{(2(z-z_0))^n\over n!(2\Lambda-1)^n}\right\vert\,\vert\vert g\vert\vert_\infty^n\,\vert\vert u-v\vert\vert_\infty.
\end{equation}
This means that, for bounded $z$, the operators ${\bf T}^n$ and $\tilde{\bf T}^n$ are contractive for large enough $n$.  From \cite[pp. 26, Theorem 3.1]{bailey} we have that the sequence $w_{n+1}^+=[{\bf T}w_n^+]$, $n=0,1,2,\ldots$, $w_0^+=0$, converges, for any $z\in{\cal L}$ bounded, to the unique solution $w_+(z)$ of problem \eqref{problemunobis} and the sequence $\tilde w_{n+1}^-=[\tilde{\bf T}\tilde w_n^-]$, $n=0,1,2,\ldots$, $\tilde w_0^-=0$, converges, for any $z\in{\cal L}$ bounded, to the unique solution $w_-(z)$ of problem \eqref{problemdosbis} multiplied by $z^{2\Lambda-1}$. Or equivalently, the sequence $y^+_n:=w^+_n+\phi_+$, that is,
\begin{equation}\label{recumas}
y^+_{n+1}(z)=y_0+{z\over 2\Lambda-1}\int_0^1\left[1-t^{2\Lambda-1}\right]g(zt)y^+_n(zt)dt, \qquad y^+_0(z)=y_0,
\end{equation}
converges, for $z\in{\cal L}$ bounded, to the unique solution $y_+(z)$ of \eqref{problemuno}. And the sequence $y^-_n:=w^-_n+\phi_-$, with $w^-_n:=z^{1-2\Lambda}\tilde w^-_n$, that is,
\begin{equation}\label{recumenos}
y^-_{n+1}(z)=\phi_-(z)+{1\over 2\Lambda-1}\int_{z_0}^z\left[1-\left({t\over z}\right)^{2\Lambda-1}\right]g(t)y^-_n(t)dt, \qquad  y^-_0(z)=\phi_-(z),
\end{equation}
converges, for $z\in{\cal L}$ bounded, to the unique solution $y_-(z)$ of \eqref{problemdos}.

Let's define the remainder of the approximation by $R^\pm_n(z):=y_\pm(z)-y^\pm_n(z)$. Setting $v(z)=w_+(z)$ and $u(z)=w^+_0(z)=0$ in \eqref{tn} and using that $[{\bf T}^n w_+]=w_+$ and $[{\bf T}^n w^+_0]=w^+_n$ or setting $v(z)=w_-(z)$ and $u(z)=w^-_0(z)=0$ in \eqref{tn} and using that $[\tilde{\bf T}^n \tilde w_-]=\tilde w_-$ and $[\tilde{\bf T}^n \tilde w^-_0]=\tilde w^-_n$ we find
$$
\vert w_\pm(z)-w^\pm_n(z)\vert\le{\vert\vert g\vert\vert_\infty^n\vert [2(z-z_0)]^n\vert\over n!\vert 2\Lambda-1\vert^n}\vert\vert w_\pm\vert\vert_\infty.
$$
In this formula and formulas below involving $w^+$ or $y^+$ (not $w^-$ or $y^-$) we must set $z_0=0$.
Using that $y_\pm(z)=w_\pm(z)+\phi_\pm(z)$ and $y^\pm_n(z)=w^\pm_n(z)+\phi_\pm(z)$ we find that the remainder $R^\pm_n(z)$ is bounded by
\begin{equation}
\label{kk}
\vert R^\pm_n(z)\vert\le{\vert\vert g\vert\vert_\infty^n\vert [2(z-z_0)]^n\vert\over n!\vert 2\Lambda-1\vert^n}\vert\vert y_\pm-\phi_\pm\vert\vert_\infty.
\end{equation}
Moreover, we have that, for problem \eqref{problemuno},
$$
y^+_{n+1}(z)-y^+_n(z)={z\over 2\Lambda-1}\int_0^1\left[1-t^{2\Lambda-1}\right]g(zt)[y^+_n(zt)-y^+_{n-1}(zt)]dt,
$$
and, for problem \eqref{problemdos},
$$
y^-_{n+1}(z)-y^-_n(z)={1\over 2\Lambda-1}\int_{z_0}^z\left[1-\left({t\over z}\right)^{2\Lambda-1}\right]g(t)[y^-_n(t)-y^-_{n-1}(t)]dt.
$$
Then, for any problem,
$$
\vert\vert y^\pm_{n+1}-y^\pm_n\vert\vert_\infty\le{2\vert z-z_0\vert\,\vert\vert g\vert\vert_\infty\over \vert 2\Lambda-1\vert}\,\vert\vert y^\pm_n-y^\pm_{n-1}\vert\vert_\infty.
$$
This means that the expansion
$$
y^\pm(z)=\phi_\pm+\sum_{k=0}^{n-1}[y^\pm_{k+1}(z)-y^\pm_k(z)]+R^\pm_n(z)
$$
is an asymptotic expansion for large $\Lambda$ and bounded $z\in{\cal L}$.

We see from \eqref{recumas} that the sequence $y^+_n(z)$ is a sequence of analytic functions in ${\cal D}$. A sequence of analytic functions that converges uniformly in any compact contained in ${\cal D}$, that is, the unique solution $y_+(z)$ of problem \eqref{problemuno} is analytic in ${\cal D}$. Analogously, the sequence $y^-_n(z)$ in \eqref{recumenos} is a sequence of analytic functions in ${\cal D}$ with, possibly, a branch point at $z=0$. This means that the unique solution $y_-(z)$ of problem \eqref{problemdos} is analytic in ${\cal D}$ except, possibly, for a branch point at $z=0$.

\begin{obs} When $g(z)$ is not analytic in ${\cal D}$, but only continuous, from the above derivation we still see that, problems \eqref{problemuno} and \eqref{problemdos} have a unique solution and the recurrences \eqref{recumas} and \eqref{recumenos} converge to the respective solutions.
\end{obs}

\begin{obs}  When $g(z)$ is an elementary function (analytic or not in ${\cal D}$), the successive approximations $y_n$ of the unique solution of those problems are iterated integrals of elementary functions.
\end{obs}

\section{The nonlinear case}

The technique used in the previous section may be easily generalized to nonlinear problems of the form
\begin{equation}\label{nonlinolvereq}
u''-{\tilde\Lambda^2\over z^2}u=\tilde f(z,u),\qquad\tilde\Lambda\to\infty,
\end{equation}
where the function $\tilde f(z,u)$ is continuous for $(z,y)\in{\cal D}\times\mathbb{C}$ and satisfies the following Lipschitz condition in its second variable:
\begin{equation}\label{lip}
\vert \tilde f(z,u)-\tilde f(z,v)\vert\le {L\over z}\vert u-v\vert, \quad  \forall\, u,v\in\mathbb{C}\;\, \text{and  $z\in{\cal D}$},
\end{equation}
with $L$ a positive constant independent of $z, u, v$.

After the change of unknown:  $u\to y:=z^{-\Lambda}u$, with  the parameter $\Lambda$ defined in \eqref{lambda}, the new unknown $y$ satisfies the nonlinear differential equation
\begin{equation}\label{nonlinecuacion}
zy''(z)+2\Lambda y'(z)=f(z,y(z)),
\end{equation}
where $f(z,y):= z^{1-\Lambda}\tilde f(z,z^\Lambda y)$.
Then, two possible well-posed problems, each of one provides a unique solution of the equation \eqref{nonlinecuacion}, are
\begin{equation}\label{nonlinproblemuno}
\begin{cases}
zy''(z)+2\Lambda y'(z)=f(z,y(z))\quad \text{in $\cal D$}, \\
y(0)=y_0,
\end{cases}
\end{equation}
and
\begin{equation}\label{nonlinproblemdos}
\begin{cases}
zy''(z)+2\Lambda y'(z)=f(z,y(z))\quad \text{in $\cal D$}, \\
y(z_0)=\bar y_0, \quad y'(z_0)=y_1,
\end{cases}
\end{equation}
where $z_0\ne 0$, $y_0=\mathcal{O}(1)$, $\bar y_0=\mathcal{O}(1)$ and $y_1=\mathcal{O}(\Lambda)$ are complex numbers.

A slight modification of the analysis of Section 3 provides, for problems \eqref{nonlinproblemuno} and \eqref{nonlinproblemdos} the same conclusions that we derived for problems \eqref{problemuno} and \eqref{problemdos}. We state them in the form of a theorem.

\begin{theorem}
Let $f:{\cal D}\times\mathbb{C}\to\mathbb{C}$ continuous and satisfy \eqref{lip}. Then, problems \eqref{nonlinproblemuno} and \eqref{nonlinproblemdos} have unique solutions that we denote by $y_+(z)$ and $y_-(z)$ respectively. They are independent whenever $(y_+(z_0),y_+'(z_0))\ne(y_-(z_0),y_-'(z_0))$. Moreover:
\begin{itemize}
\item[(i)] For $n=0,1,2,\ldots$, the sequences
\begin{equation}\label{nonlinrecumas}
y^+_{n+1}(z)=y_0+{z\over 2\Lambda-1}\int_0^1\left[1-t^{2\Lambda-1}\right]f\left(tz,y^+_n(zt)\right)dt, \qquad y^+_0(z)=y_0,
\end{equation}
\begin{equation}\label{nonlinrecumenos}
y^-_{n+1}(z)=\phi_-(z)+{1\over 2\Lambda-1}\int_{z_0}^z\left[1-\left({t\over z}\right)^{2\Lambda-1}\right]f\left(t,y^-_n(t)\right)dt, \qquad  y^-_0(z)=\phi_-(z),
\end{equation}
with $\phi_-(z)$ defined in \eqref{fi}, converge, for $z\in{\cal L}$ bounded, to the unique solutions $y_+(z)$ of \eqref{nonlinproblemuno} and $y_-(z)$ of \eqref{nonlinproblemdos} respectively.
\item[(ii)] The remainder $R^{\pm}_n(z):=y_\pm (z)-y^{\pm}_n(z)$ is bounded by
\begin{equation}\label{cotanonlin}
\vert R^\pm_n(z)\vert\le{L^n\vert [2(z-z_0)]^n\vert\over n!\vert 2\Lambda-1\vert^n}\vert\vert y_\pm-\phi_\pm\vert\vert_\infty.
\end{equation}
And, in consequence, the expansion
$$
y^\pm(z)=\phi_\pm+\sum_{k=0}^{n-1}[y^\pm_{k+1}(z)-y^\pm_k(z)]+R^\pm_n(z)
$$
is an asymptotic expansion for large $\Lambda$ and bounded $z\in{\cal L}$.
\end{itemize}
\end{theorem}

\begin{proof}
It is similar to the analysis of the previous section. Then, we only give here a few significant details. After the change of unknown $y_\pm(z)$ $\to$ $w_\pm(z):=y_\pm(z)-\phi_\pm(z)$, problems \eqref{nonlinproblemuno}, \eqref{nonlinproblemdos} read, respectively
\begin{equation}\label{nonlinproblemunoi}
\begin{cases}
zw_+''(z)+2\Lambda w_+'(z)=F_+(z,w_+):=f(z,w_+(z)+\phi_+(z))\quad \text{in $\cal D$}, \\
w_+(0)=0,
\end{cases}
\end{equation}
and
\begin{equation}\label{nonlinproblemdosi}
\begin{cases}
zw_-''(z)+2\Lambda w_-'(z)=F_-(z,w_-):=f(z,w_-(z)+\phi_-(z))\quad \text{in $\cal D$}, \\
w_-(z_0)=w_-'(z_0)=0.
\end{cases}
\end{equation}
The solutions of these problems satisfy the Volterra integral equations of the second kind $w_+(z)=[{\bf T}w_+](z)$ and $w_-(z)=[{\bf T}w_-](z)$ where now, the operator ${\bf T}$ is nonlinear, and defined by
\begin{equation}\nonumber
[{\bf T}w_\pm](z):= {1\over 2\Lambda-1}\int_{z_0}^z\left[1-\left({t\over z}\right)^{2\Lambda-1}\right]f\left(t,w_\pm(t)+\phi_\pm(t)\right)dt,
\end{equation}
where $z_0$ must be set equal to zero for $w_+$. From \eqref{lip} we have the Lipschitz condition
\begin{equation}\label{lipbis}
\vert  f(z,u)- f(z,v)\vert\le {L}\vert u-v\vert \qquad  \forall\, u,v\in\mathbb{C}\;\, \text{and  $z\in{\cal D}$},
\end{equation}
with $L$ given in \eqref{lip}. From here, and using \eqref{lipbis} the proof is identical to the one of the previous section replacing $\norm{g}_\infty$ by $L$.
\end{proof}

\section{Olver's method for equation \eqref{ecuacion}}

In this section we consider two (at this moment unknown) independent solutions $y_+(z)$ and $y_-(z)$ of \eqref{ecuacion} and propose the following representations in the form of formal asymptotic expansions for large $\Lambda$:
\begin{equation}\label{olverunoa}
y_\pm(z)=y_n^\pm(z)+R_n^\pm(z),
\end{equation}
with
\begin{equation}\label{olverunob}
y_n^+(z):=\sum_{k=0}^{n-1}{A_k(z)\over (2\Lambda )^k},
\qquad y_n^-(z):=z^{1-2\Lambda}\sum_{k=0}^{n-1}{A_k(z)\over [2(1-\Lambda)]^k},
\end{equation}
and the obvious definition of $R_n^\pm(z)$.
When we introduce \eqref{olverunoa} and \eqref{olverunob} in the equation $zy''+2\Lambda y'=gy$ we find that both, $y_+(z)$ and $y_-(z)$, formally satisfy the respective differential equations, term-wise in $(2\Lambda )^k$ or $[2(1-\Lambda )]^k$, if, for $n=0,1,2,\ldots$,
\begin{equation}\label{olverecu}
A_{n+1}(z)=A_n(z)- zA_n'(z)+\int g(z)A_n(z)dz
\end{equation}
and
\begin{equation}\nonumber
\begin{split}
z[R_n^+(z)]''+2\Lambda [R_n^+(z)]'=&\, {A_n'(z)\over(2\Lambda )^{n-1}}+g(z)R_n^+(z), \\
z[R_n^-(z)]''+2(1-\Lambda )[R_n^-(z)]'=&\, {A_n'(z)\over[2(1-\Lambda )]^{n-1}}+g(z)R_n^-(z).
\end{split}
\end{equation}
Without loss of generality we may fix $A_0(z)=1$. It is obvious that, when $g(z)$ is analytic in ${\cal D}$, for $n=0,1,2,\ldots$, the coefficients $A_n(z)$ are analytic in ${\cal D}$ too. When $g(z)$ is only continuous in ${\cal L}$, for $n=0,1,2,\ldots$, the coefficients $A_n(z)$ are continuous in ${\cal L}$ and bounded at $z=0$.

We seek a solution $y^+(z)$ regular at $z=0$ and a solution $y^-(z)$ regular at $z=z_0\ne 0$. Therefore, without loss of generality, we may set $R_n^+(0)=0$ and $R_n^-(z_0)=[R_n^-]'(z_0)=0$. Then, these remainders are solutions of the respective initial value problems:
\begin{equation}\nonumber
\begin{cases}
z[R_n^+(z)]''+2\Lambda [R_n^+(z)]'= \displaystyle{A_n'(z)\over(2\Lambda )^{n-1}}+g(z)R_n^+(z) \quad \text{in ${\cal D}$}, \\
R_n^+(0)= 0,
\end{cases}
\end{equation}
and
\begin{equation}\nonumber
\begin{cases}z[R_n^-(z)]''+2(1-\Lambda) [R_n^-(z)]'= \displaystyle{A_n'(z)\over[2(1-\Lambda )]^{n-1}}+g(z)R_n^-(z)\quad \text{in ${\cal D}$}, \\
R_n^-(z_0)=[R_n^-]'(z_0)=0.
\end{cases}
\end{equation}

The first problem for $R_n^+(z)$ is identical to problem \eqref{problemunobis} for $w^+(z)$ replacing $g(z)\phi_+(z)$ by $A_n'(z)/(2\Lambda )^{n-1}$. The second problem for $R_n^-(z)$ is identical to problem \eqref{problemdosbis} for $w^-(z)$ replacing $g(z)\phi_-(z)$ by $A_n'(z)/(2\Lambda)^{n-1}$ and then $\Lambda$ by $1-\Lambda$.  Therefore, proceeding as in Section 3 we find that $R_n^+(z)$ and $R_n^-(z)$ are solutions of the respective Volterra integral equations
$$
R_n^+(z)={1\over 2\Lambda -1}\int_0^z\left[1-\left({t\over z}\right)^{2\Lambda -1}\right]\left[{A_n'(t)\over(2\Lambda )^{n-1}}+g(t)R_n^+(t)\right]dt,
$$
$$
R_n^-(z)={1\over 1-2\Lambda}\int_{z_0}^z\left[1-\left({z\over t}\right)^{2\Lambda-1}\right]\left[{A_n'(t)\over[2(1-\Lambda )]^{n-1}}+g(t)R_n^-(t)\right]dt.
$$
Using that $\vert 1-(t/z)^{2\Lambda -1}\vert\le 2$ for $t\in[0,z]$ and $\vert 1-(z/t)^{2\Lambda-1}\vert\le 2$  for $t\in[z,z_0]$, we derive the bound
$$
\vert R_n^-(z)\vert\le{2\over \vert 2\Lambda -1\vert}\int_{z_0}^z\vert g(t)R_n^-(t)\vert \vert dt\vert+{2\over \vert 2\Lambda -1\vert}\int_{z_0}^z\left\vert{A_n'(t)\over[2(1-\Lambda )]^{n-1}}\right\vert\vert dt\vert
$$
and the same bound for $R^+_n(z)$ replacing $\Lambda$ by $1-\Lambda$ and setting $z_0=0$.
Applying Gronwall's lemma \cite{coddington} we obtain
$$
\vert R_n^-(z)\vert\le \frac{2 e^{{2\over\vert 2\Lambda -1\vert}\int_{z_0}^z \vert g(t)\vert\vert dt\vert} }{\vert (2\Lambda -1)[2(1-\Lambda )]^{n-1}\vert}\int_{z_0}^z \vert A_n'(t)\vert\vert dt\vert
$$
and  the same bound for $R^+_n(z)$ replacing $\Lambda$ by $1-\Lambda$ and setting $z_0=0$.
When $A_n'(t)$ and $g(t)$ are integrable in ${\cal L}$ (this is granted when ${\cal L}$ is bounded), we also have the bounds:
\begin{equation}\label{kkk}
\begin{split}
\vert R_n^+(z)\vert\le {2\vert\vert A_n'\vert\vert_1\over\vert (2\Lambda -1)(2\Lambda )^{n-1}\vert}e^{2\vert\vert g\vert\vert_1/\vert 2\Lambda -1\vert}, \\
\vert R_n^-(z)\vert\le {2\vert\vert A_n'\vert\vert_1\over\vert (2\Lambda -1)[2(1-\Lambda )]^{n-1}\vert}e^{2\vert\vert g\vert\vert_1/\vert 2\Lambda -1\vert},
\end{split}
\end{equation}
where
$$
\vert\vert g\vert\vert_1:=\int_{\cal L} \vert g(t)\vert\vert dt\vert, \qquad\vert\vert A_n'\vert\vert_1:=\int_{\cal L} \vert A_n'(t)\vert\vert dt\vert.
$$
These bounds show the asymptotic character of the expansions \eqref{olverunoa}.

\begin{obs} We see from \eqref{olverecu} that the coefficients $A_n(z)$ are analytic functions in ${\cal D}$. Therefore, the asymptotic approximation $y_n^+(z)$ to the unique solution $y_+(z)$ of problem \eqref{problemuno} and the asymptotic approximation $y_n^-(z)$ to the unique solution $y_-(z)$ of problem \eqref{problemdos} are analytic in ${\cal D}$.
\end{obs}

\section{Example and numerical experiments}

Consider the differential equation
\begin{equation}\label{ejem}
z y''(z)+2\Lambda y'(z)=y(z).
\end{equation}
To find asymptotic approximations for large $\Lambda$ of two independent solutions of this equation we consider the two associated initial value problems:
\begin{equation}\label{ejem1}
\begin{cases}
z y''(z)+2\Lambda y'(z)=y(z)\quad \text{in $\mathbb{C}$,} \\
y(0)=1,
\end{cases}
\end{equation}
and
\begin{equation}\label{ejem2}
\begin{cases}z y''(z)+2\Lambda y'(z)=y(z)\quad \text{in $\mathbb{C}$,} \\
y(1)=K_{2\Lambda-1}(2),\quad y'(1)=-K_{2\Lambda}(2).
\end{cases}
\end{equation}
The unique solution of \eqref{ejem1} is a modified Bessel function (analytic in $\mathbb{C}$)
$$
y(z)=\Gamma(2\Lambda)z^{1/2-\Lambda}I_{2\Lambda-1}(2\sqrt{z}),
$$
and the unique solution of \eqref{ejem2} is a modified Bessel function
$$
y(z)=z^{1/2-\Lambda}K_{2\Lambda-1}(2\sqrt{z}),
$$
analytic in $\mathbb{C}\setminus\mathbb{R}^-$.


The iterative method introduced in Section 3 provides a convergent as well as an asymptotic expansion of these functions for large $\Lambda$ in terms of elementary functions. The recurrence relation \eqref{recumas} for problem \eqref{ejem1} is given by
\begin{equation}\label{recu1}
\begin{split}
&y^+_0(z)=1,\\
&y^+_{n+1}(z)=1+{z\over 2\Lambda-1}\int_0^1\left[1-t^{2\Lambda-1}\right]y^+_n(zt)dt,
\end{split}
\end{equation}
and the recurrence relation \eqref{recumenos} for problem \eqref{ejem2} is defined by
\begin{equation}\label{recu2}
\begin{split}
& y^-_0(z)=K_{1-2\Lambda}(2)-\frac{K_{2\Lambda}(2)}{1-2\Lambda}\left(z^{1-2\Lambda}-1\right),\\
&y^-_{n+1}(z)=y^-_0(z)+{1\over 2\Lambda-1}\int_{z_0}^z\left[1-\left({t\over z}\right)^{2\Lambda-1}\right]y^-_n(t)dt.
\end{split}
\end{equation}

On the other hand, applying Olver's method, the general solution of the differential equation \eqref{ejem} is given by a linear combination of the solutions $y_+$ and $y_-$ of problems \eqref{ejem1} and \eqref{ejem2} respectively. An asymptotic expansion of the unique solution of problem \eqref{ejem1} is proportional to
$$
y_n^+(z):=\sum_{k=0}^{n-1}{A_k(z)\over (2\Lambda )^k}.
$$
An asymptotic expansion of the unique solution of problem \eqref{ejem2} is a linear combination of
$$
y_n^+(z):=\sum_{k=0}^{n-1}{A_k(z)\over (2\Lambda )^k}
\quad \text{and} \quad y_n^-(z):=z^{1-2\Lambda}\sum_{k=0}^{n-1}{A_k(z)\over [2(1-\Lambda)]^k}.
$$
In all these formulas the coefficients $A_n(z)$ are given by the recurrence \eqref{olverecu} with $g(z)=1$:
$$
\begin{cases}
A_0(z)=1,\\
A_{n+1}(z)=A_n(z)- zA_n'(z)+\displaystyle\int A_n(z)dz.
\end{cases}
$$
They are polynomials in the variable $z$:
\begin{equation}\nonumber
\begin{split}
& A_0(z)=1,\quad
 A_1(z)=1+z, \quad A_2(z)=\displaystyle 1+z+\frac{z^2}{2},
\quad A_3(z)=\displaystyle 1+z+\frac{z^3}{6},\\
& A_4(z)=\displaystyle 1+z+\frac{z^2}{2}-\frac{z^3}{3}+\frac{z^4}{24},\quad
 A_5(z)=\displaystyle 1+z+\frac{5 z^3}{6}-\frac{5 z^4}{24}+\frac{z^5}{120},\quad \ldots.
\end{split}
\end{equation}

Olver's method also gives an asymptotic expansion of the unique solution $y_+(z)$ of \eqref{ejem1} and the unique solution $y_-(z)$ of \eqref{ejem2} for large $\Lambda$ in terms of elementary functions of $z$.

Table \ref{tabla1} and Table \ref{tabla2} show some numerical approximations, for different values of $z$ and $\Lambda$, of the solutions of \eqref{ejem1} and \eqref{ejem2} respectively supplied by the iterative algorithm compared with the approximation supplied by Olver's method.

\begin{table}[h]{\setlength{\extrarowheight}{3.5pt}\hspace*{-0.5cm} \small
\begin{center}{\hspace*{-0.25cm}
\begin{tabular}{|c|c|c|c|}
\hline
\multicolumn{4}{c}{$z=1$}
\\ \hline \hline $\Lambda$ & $n$ & Olver's method & Formula \eqref{recu1}
\\ \hline
& 1 & 0.22798242 & 0.080931451
\\
$0.75$ & 3 & 0.06396403 & 0.00040353
\\
 & 5 & 0.01879412 & 3.69$e-$7
\\ \hline
& 1 & 0.01246076 & 0.00423127
\\
$5$ & 3 & 0.00010294 & 2.22$e-$6
\\
 & 5 & 5.66$e-$7 & 3.52$e-$10
\\ \hline
& 1 & 0.00003714 & 0.00001239\\
$100$ & 3 & 7.49$e-$10 & 2.51$e-$11
\\
 & 5 & 9.25$e-15$ & 2.00$e-$17
\\ \hline
& 1 & 1.49$e-$6 & 4.99$e-$7
\\
$500$ & 3 & 1.20$e-$12 & 4.13$e-$14
\\
 & 5 & 5.93$e-19$ & 1.0$e-$21
\\ \hline
\end{tabular} \quad\begin{tabular}{|c|c|c|c|}
\hline
\multicolumn{4}{c}{$z=-2$}
\\ \hline\hline  $\Lambda$ & $n$ & Olver's method & Formula \eqref{recu1}
\\ \hline
& 1 & 1.00000000 & 4.08781323
\\
$0.5$ & 3 & 0.69593774 & 0.13062516
\\
 & 5 & 6.00987600 & 0.00060326
\\ \hline
& 1 & 0.00118982 & 0.02105883
\\
$5$ & 3 & 0.00027873 & 0.00004621
\\
 & 5 & 0.00002455 & 2.96$e-$8
\\ \hline
& 1 & 1.31$e-$6 & 0.00020038
\\
$50-2i$ & 3 & 3.25$e-$8 & 6.36$e-$9
\\
 & 5 & 2.8$e-$11 & 1.1$e-$13
\\ \hline
& 1 & 1.65$e-$7 & 0.00005008
\\
$100$ & 3 & 2.06$e-$9 & 4.07$e-$10
\\
 & 5 & 5.0$e-$13 & 1.0$e-$14
\\ \hline
\end{tabular}}
\caption{\label{tabla1} \small Numerical experiments about the relative errors in the approximation of the solution of problem \eqref{ejem1} using Olver's method and the iterative method \eqref{recu1} for different values of $\Lambda$ and $n$.}\end{center}}
\end{table}

\begin{table}[h!]{\setlength{\extrarowheight}{3.5pt}\hspace*{-0.5cm} \small
\begin{center}{\hspace*{-0.25cm}
\begin{tabular}{|c|c|c|c|}
\hline
\multicolumn{4}{c}{$z=0.5$}
\\ \hline \hline $\Lambda$ & $n$ & Olver's method & Formula \eqref{recu2}
\\ \hline
& 1 & 0.11724359 & 0.00308515
\\
$0.75$ & 3 & 0.15072603 & 2.04$e-$7
\\
 & 5 & 0.22999718 & 1.98$e-$12
\\ \hline
& 1 & 0.04701568 & 0.00080406
\\
$5$ & 3 & 0.00120818 & 3.56$e-$8
\\
 & 5 & 0.00003105 & 2.86$e-$13
\\ \hline
& 1 & 0.00974880 & 0.00004491\\
$25+5i$ & 3 & 5.46$e-$6 & 2.66$e-$10
\\
 & 5 & 3.67$e-9$ & 3.78$e-$14
\\ \hline
& 1 & 0.00498611 & 0.00001207\\
$50$ & 3 & 6.85$e-$7 & 2.15$e-$11
\\
 & 5 & 1.15$e-10$ & 7.93$e-$15
\\ \hline
\end{tabular} \quad\begin{tabular}{|c|c|c|c|}
\hline
\multicolumn{4}{c}{$z=-1+i/4$}
\\ \hline\hline  $\Lambda$ & $n$ & Olver's method & Formula \eqref{recu2}
\\ \hline
& 1 & 1.06271455 & 0.50808214
\\
$0.75$ & 3 & 0.87941096 & 0.01338941
\\
 & 5 & 0.91915445 & 0.00005423
\\ \hline
& 1 & 0.21929092 & 0.02356432
\\
$5$ & 3 & 0.00507404 & 0.00013029
\\
 & 5 & 0.00013229 & 4.81$e-$7
\\ \hline
& 1 & 0.03935288 & 0.00089998
\\
$25$ & 3 & 0.00002050 & 1.38$e-$7
\\
 & 5 & 1.39$e-$8 & 8.51$e-$12
\\ \hline
& 1 & 0.01924853 & 0.00023144
\\
$50$ & 3 & 2.35$e-$6 & 9.05$e-$9
\\
 & 5 & 3.85$e-$10 & 1.44$e-$13
\\ \hline
\end{tabular}}
\caption{\label{tabla2} \small Numerical experiments about the relative errors in the approximation of the solution of problem \eqref{ejem2} using Olver's method and the iterative method \eqref{recu2} for different values of $\Lambda$ and $n$.}\end{center}}
\end{table}

\section{Final remarks}

Olver's asymptotic expansion \eqref{expan} fails for $m=2$. In Section 2 we have modified the differential equation in the case $m=2$ that moves the asymptotic parameter $\Lambda$ from the coefficient of the unknown $u$ in the original differential equation to the coefficient of the derivative $y'$ in the new differential equation. Then, we have proposed two methods to obtain asymptotic expansions of two independent solutions of this equation: one method is just Olver's idea applied to the new differential equation. The other method is a fixed point technique that gives an asymptotic expansion for large $\Lambda$ that is also convergent. Moreover, this second method can be also applied to nonlinear differential equations. For $m\ne 2$ the asymptotic behavior for large $\Lambda$ of the solutions of \eqref{olvereq} is exponential. As a difference with the cases $m\ne 2$, in the case $m=2$ the asymptotic behavior of the solutions is not exponential, but of power type. This is why the standard Olver's method cannot be directly applied in this case.

The approximations $y_n^+(z)$ and $y_n^-(z)$ to two independent solutions of the equation \eqref{ecuacion}, derived with Olver's method, are analytic in ${\cal D}$ when $g(z)$ is analytic. On the other hand, the approximation $y_n^+(z)$ derived with the fixed point method is analytic in ${\cal D}$, but the approximation $y_n^-(z)$, is analytic in ${\cal D}$ except, possibly, for a branch point at $z=0$. In fact, the solution of \eqref{problemuno} is analytic in ${\cal D}$, whereas the solution of \eqref{problemdos} is analytic in ${\cal D}$ except, possibly, for a branch point at $z=0$. The difference between the approximations given by Olver's method and the approximations given by the fixed point method is that the later are convergent, whereas the former, in general, are not. Then, the analytic properties of the solution are the same as the analytic properties of the approximants of the fixed point method. In Olver's approximation, whereas the remainder $R_n^+(z)$ is analytic in ${\cal D}$, the remainder $R_n^-(z)$ is analytic in ${\cal D}$ except, possibly, for a branch point at $z=0$.

We start the sequence \eqref{recumas} at $y^+_0(z)=y_0$, a function bounded at $z=0$. We observe in \eqref{recumas} that the iteration $y_n^+\to y^+_{n+1}$ keeps this property, as all the terms of the sequence $y^+_n$ are bounded at $z=0$. And the sequence converges to a function of the unique one-dimensional space of solutions of equation \eqref{ecuacion} that are bounded at $z=0$. The situation is different with the recurrence \eqref{recumenos}. Except for the above mentioned one-dimensional space, the whole two-dimensional space of solutions of the equation \eqref{ecuacion} consists of functions unbounded at $z=0$. Then, even if we start the sequence $y^-_n(z)$ with a function $y^-_0(z)$ analytic at $z=0$, that is, if we take $y_1=0$ and $y_0\ne 0$ in \eqref{recumenos}, the iteration $y^-_n\to y^-_{n+1}$, in general, does not keep this property, it falls off the one dimensional space of bounded solutions at $z=0$.

The situation described in the above paragraph is one side of the coin. The other side is the fact that, for the equation \eqref{ecuacion}, it is possible to get asymptotic approximations for the unique solution of an initial value problem with initial data prescribed at $z=0$: problem \eqref{problemuno}, using either the fixed point technique or Olver's method. These methods do not work when we want to approximate a second solution independent of the previous one using an initial value problem with initial data prescribed at $z=0$: observe that we cannot set $z_0=0$ in the recursion \eqref{recumenos} as the integrals become meaningless. Something similar occurs in Olver's method: we cannot find a bound for the remainder $R_n^-(z)$ if we set $z_0=0$, as the kernel $1-(z/t)^{2\Lambda-1}$ is not bounded for $t\in[0,z]$. That is why we have considered the initial value problem \eqref{problemdos} with $z_0\ne 0$.

The error bounds \eqref{tn} and \eqref{kk} are not uniform in $z$. This means that the convergent and asymptotic character of the expansions of Section 3 is proved only over bounded subsets of ${\cal D}$. On the other hand, when $A_n'$ and $g$ are integrable in unbounded paths ${\cal L}$, the bound \eqref{kkk} shows the uniform character of the Olver's asymptotic expansions of Section 5.

\section*{Acknowledgments}

The {\it Direcci\'on General de Ciencia y Tecnolog\'{\i}a} (REF. MTM2010-21037) is acknowledged by its financial support.

\footnotesize{
}
\end{document}